\documentclass[11pt,a4paper]{article}
\author{Steffen Lauritzen\\ University of Copenhagen}
\title{Total Variation Convergence\\ Preserves Conditional Independence}

\usepackage{amssymb,amsmath,amsthm}
\usepackage{natbib}

\newtheorem{theorem}{Theorem}
\newtheorem{corr}{Corollary}

\newcommand{\dd}{\mathrm{d}}

\newcommand{\cd}{\,|\,}
\newcommand{\cip}{{\perp\!\!\!\perp}}

\newcommand{\E}{{\mathbf E}}
\newcommand{\FF}{\mathbb{F}}
\renewcommand{\AA}{\mathbb{A}}
\newcommand{\BB}{\mathbb{B}}

\begin{document}

\maketitle




\begin{abstract}This note establishes that if a sequence $P_n, n=1,\ldots$ of probability measures  converges in total variation to the limiting probability measure  $P$, and $\sigma$-algebras $\mathbb{A}$ and $\mathbb{B}$ are conditionally independent given $\mathbb{H}$ with respect to $P_n$ for all $n$, then they are also conditionally independent with respect to the limiting measure $P$. As a corollary, this also extends to pointwise convergence of densities to a density.

\end{abstract}

%


\section{Introduction}
\citet[pp.\ 38--39]{lauritzen:96} shows by counterexample that conditional independence is not closed under weak convergence in general, whereas this is the case for discrete sample spaces. However, as we shall see below, strengthening the convergence to convergence in total variation implies that conditional independence is preserved. 
\section{Preliminaries}
We consider a general probability space $(\Omega,
\FF, P)$, where $\Omega$ is a sample space, $\FF$ a $\sigma$-algebra of subsets of $\Omega$ and $P$ a $\sigma$-additive probability measure on $\FF$. 
For $\sigma$-algebras $\mathbb{A}$ and $\mathbb{B}$ 
we let $\mathbb{A}\vee  \mathbb{B}$ denote the smallest $\sigma$--algebra that contains both $\mathbb{A}$ and $\mathbb{B}$ and note that this is generated by the intersections$$ \AA\vee \BB = \sigma
\{A\cap B\cd A\in\mathbb{A}, B\in\mathbb{B}\}.
$$

Then the conditional probability of a set $A\in \AA$ given a $\sigma$-algebra $\mathbb{H}$ is defined as the conditional expectation 
of the indicator function $\mathbf{1}_A$:
\begin{equation*}
P(A \; \cd \mathbb{H}) = \E( \mathbf{1}_A \cd \mathbb{H}) 
\end{equation*}
and we say that 
two $\sigma$-algebras, $\mathbb{A}, \mathbb{B}\subseteq \FF$ are
conditionally independent given a $\sigma$-algebra $\mathbb{H}\subseteq \FF$ if
$$
P (A\cap B\cd \mathbb{H})=P(A \cd \mathbb{H}) P(B\cd \mathbb{H}) \text{ a.s.\ } P\quad  \mbox{for all\ } A \in \mathbb{A}, B
\in \mathbb{B} \, .
$$
Symbolically, we will write $\mathbb{A} \cip \mathbb{B} 
\cd \mathbb{H}$
and recall \citep{dawid:80} 
that 
\begin{equation}\label{eq:cipcondition}\mathbb{A} \cip \mathbb{B} \cd \mathbb{H}\iff
P(A \cd \mathbb{B} \vee \mathbb{H})= P(A\cd \mathbb{H}) \text{ a.s.\ } P.\end{equation}

\section{The main result}
\begin{theorem}\label{thm:totvar_conv}Let $P_n,n=1,\ldots$ be a sequence of probability measures on $(\Omega,\mathbb{F})$  that converges in total variation to a probability measure $P$ and let $\mathbb{A}$, $\mathbb{B}$, and $\mathbb{H}$ be sub-$\sigma$-algebras of $\mathbb{F}$. If $\mathbb{A}\cip \mathbb{B}\cd \mathbb{H}$ with respect to $P_n$ for all $n$, then  $\mathbb{A}\cip \mathbb{B}\cd \mathbb{H}$ with respect to $P$.
\end{theorem}
\begin{proof}Let $A\in \mathbb{A}$, $B\in \mathbb{B}$ and $H\in \mathbb{H}$ and let $$Z_n^A=P_n(A\cd \mathbb{H})= P_n(A\cd \mathbb{B}\vee \mathbb{H})$$ where we note that $Z_n^A$ is $\mathbb{H}$-measurable for all $n$. 
Since $Z_n^A, n=1,\ldots $ is bounded in $L_2(P)$, it has a subsequence $Z_{n_j}^A, j=1, \ldots$ that converges weakly in $L_2(P)$ to $Z^A_\infty$, i.e.\  in particular  
$$\lim_{j\to\infty} \int_{B\cap H} Z^A_{n_j}\,\dd P= \int_{B\cap H} Z^A_\infty \,\dd P\quad \text{for all $B\in\mathbb{B}$, $H\in \mathbb{H}$. }$$
  Further, $Z^A_\infty$ is $\mathbb{H}$ measurable. 
But since $|Z_n^A|\leq 1$ for all $n$ we have
$$  \left| \int_{B\cap H} Z^A_n \,\dd P- \int_{B\cap H} Z^A_n \,\dd P_n\right|\leq \Vert P-P_n\Vert_{\infty},$$
where $\Vert\cdot\Vert_{\infty}$ is the total variation norm
$\Vert P-Q\Vert_{\infty}=\sup_{F\in \mathbb{F}}|P(F)-Q(F)|.$ 
As $P_n$ converges in total variation to $P$  we therefore have that
\begin{eqnarray*}P(A\cap B\cap H)&=& \lim_{j\to \infty}P_{n_j}(A\cap B\cap H)=\lim_{j\to \infty} \int_{B\cap H} Z^A_{n_j} \,\dd P_{n_j}\\&=&
\lim_{j\to \infty}\int_{B\cap H} Z^A_{n_j} \,\dd P=\int_{B\cap H} Z^A_\infty \,\dd P\,.
\end{eqnarray*}
Since $\BB\vee\mathbb{H}$ is generated by the intersections $B\cap H$, we conclude that
$Z^A_\infty$ is a version of $P(A\cd \mathbb{B}\vee\mathbb{H})$ and thus  from (\ref{eq:cipcondition}) we conclude that $\mathbb{A}\cip \mathbb{B}\cd \mathbb{H}$ with respect to $P$, as desired. 
\end{proof}
In particular this gives the following corollary.
\begin{corr}Let $P_n,n=1,\ldots$ be a sequence of probability measures on $(\Omega,\mathbb{F})$  that have densities $f_n$ with respect to a $\sigma$-finite measure $\mu$. Assume  that the densities converge pointwise and almost everywhere with respect to $\mu$ to a probability density $f$. Let $P$ denote the probability measure with density $f$ with respect to $\mu$ and let $\mathbb{A}$, $\mathbb{B}$, and $\mathbb{H}$ be sub-$\sigma$-algebras of $\mathbb{F}$. If $\mathbb{A}\cip \mathbb{B}\cd \mathbb{H}$ with respect to $P_n$ for all $n$, then  $\mathbb{A}\cip \mathbb{B}\cd \mathbb{H}$ with respect to $P$.
\end{corr}
\begin{proof}This follows directly from Theorem~\ref{thm:totvar_conv} by Scheff{\'e}'s theorem \citep{scheffe:47}, see also \citet[Thm.\ 16.11]{billingsley:79}, establishing that almost everywhere convergence of densities to a density implies convergence in total variation. 
\end{proof}
Note also that Proposition 3.12 in \cite{lauritzen:96}  that weak convergence for discrete probability measures preserves conditional independence is a special instance of the corollary.
\section{Acknowledgements}
The main result was pointed out to the author by A. Klenke (personal communication) during lectures at the Summer School of Probability in Saint Flour, 2006.

\end{document}